\documentclass[11pt]{amsart}
\usepackage{amsmath,amssymb,amsthm,upref,graphicx,mathrsfs}

\usepackage{color,xcolor}
\usepackage[
  colorlinks=true,
  linkcolor=blue,
  citecolor=blue,
  urlcolor=blue]{hyperref}

\numberwithin{equation}{section}

\newtheorem{thm}{Theorem}[section]
\newtheorem{cor}[thm]{Corollary}

\numberwithin{equation}{section}

\begin{document}

\leftline{ \scriptsize}

\vspace{1.3 cm}
\title
{ Symmetry classes of tensors and Semi-direct product of finite abelian groups}
\author{ Kijti Rodtes and  Kunlathida Chimla }
\thanks{{\scriptsize
\hskip -0.4 true cm MSC(2000): Primary   20C15; Secondary 15A69
}}
\hskip -0.4 true cm

\maketitle


\begin{abstract} In the study of symmetry classes of tensors, finding examples of symmetry classes of tensors that possess an o*-basis is of considerable interest.  There are only few classes of groups that have been provided a necessary and sufficient condition for having such a basis.  There is no general criterion for any finite groups yet.  In this note, we provide a necessary and sufficient condition for the existence of o*-basis of symmetry classes of tensors associated with semi-direct product of some finite abelian groups and, consequently, their wreath product.
\end{abstract}

\vskip 0.2 true cm


\pagestyle{myheadings}
\markboth{\rightline {\scriptsize Kijti Rodtes and  Kunlathida Chimla}}
         {\leftline{\scriptsize }}
\bigskip
\bigskip


\vskip 0.4 true cm

\section{Introduction}
Many branches of pure and applied mathematics: combinatorial theory, matrix theory, operator theory, group representation theory, differential geometry, partial differential equations, quantum mechanics and other areas motivate the study of symmetry classes of tensors \cite{Chi}.   In particular, finding examples of (higher) symmetry classes of tensors that possesses a special basis (o*-basis) raised by B.Y. Wang and M.P. Gong in \cite{Gong} is of considerable interest.

	Necessary and sufficient conditions for the existence of o*-basis of symmetry classes of tensors are provided in few classes of groups;  for example, dihedral groups, \cite{HT}, dicyclic groups, \cite{DP}, semi-dihedral groups, \cite{HORO},  non abelian groups of order $pq$, \cite{kijtipq},  some subgroups of full symmetric groups and some types of p-groups in \cite{Holmm}. For any finite group, there is a necessary condition provided in \cite{shahryari}.

	It is well known that a finite group $G$ can be decomposed as a semi-direct product $G=A\rtimes_\phi H$ of a group $A$ by $H$  if and only if the short exact sequence $$ \{e\}\longrightarrow A\longrightarrow G \longrightarrow H\longrightarrow \{e\}$$
splits.  Note also that not every short exact sequence splits and then some extension groups of $A$ by $H$ can not always be written as a semi-direct product but it is a subgroup of some wreath product $A\wr_\Omega H$, for some $H$-set $\Omega$ (universal embedding theorem).  Wreath product is defined via semi-direct product and it can be used to construct a lot of interesting groups.  However, many well known groups can be written as a semi-direct product of finite abelian groups; for example, dihedral groups, semi-dihedral groups,  non-abelian groups of order $pq$ and $Z$-groups (group in which its Sylow subgroups are all cyclic and this groups have applications in the classification of finite groups, \cite{Wong}).   In this article, we provide a necessary and sufficient condition for the existence of o*-basis of symmetry classes of tensors associated with semi-direct product of some finite abelian groups and, consequently, their wreath product.
\section{Preliminary}

Let $V$ be an $n$-dimensional complex inner product space with an orthonormal basis $B=\{ e_1,\dots,e_n\}$.  Then, the $m$-folds tensor space $V^{\otimes m}:=V\otimes V \otimes \cdots\otimes V$ has $B^{\otimes m}=\{e^\otimes _\alpha\, |\, \alpha \in \Gamma_{m,n} \} $ as an orthonormal basis (with respect to the induced inner product), where $\Gamma_{m,n}$ is the set of all sequences $\alpha = (\alpha_1,\dots,\alpha_m)$, with $1  \leq \alpha_i\leq n$ and
$$ e^\otimes_\alpha=e_{\alpha_1}\otimes \cdots \otimes e_{\alpha_m}.$$
Let $G$ be a subgroup of $S_m$.  Define the action of $G$ on $\Gamma_{m,n}$ by
$$
\alpha\sigma =(\alpha_{\sigma^{-1}(1)},\dots, \alpha_{\sigma^{-1}(m)} ).
$$
The space $V^{\otimes m}$ is a left $\mathbb{C}G$-module with the action given by $\sigma e^{\otimes}=e^\otimes_{\alpha \sigma^{-1}}$ for each $\sigma \in G, \alpha \in \Gamma_{m,n}$, extended linearly.  Under this action, the induced inner product on $V^{\otimes m}$ is $G$-invariant.  Let $O(\alpha)=\{\alpha\sigma\,|\,\sigma\in G \}$ be the orbit of $\alpha$. We write $\alpha \sim \beta $ if $\alpha$ and $\beta$  belong to the same orbit in $\Gamma_{m,n}$. Let $ \Delta$ be a
system of distinct representatives of the orbits.  We denote by $G_\alpha$ the stabilizer subgroup of $\alpha$, i.e., $G_\alpha=\{\sigma\in G\,|\,\alpha\sigma=\alpha \}$.

For any $\sigma \in G$, define the operator $
P_\sigma:  V^{\otimes m} \longrightarrow  V^{\otimes m}
$ on the m-folds tensor space by
\begin{equation*}
  \label{e10}
 P_\sigma (v_1\otimes \cdots \otimes v_m )= v_{\sigma^{-1}(1)}\otimes \cdots\otimes v_{\sigma^{-1}(m)} .
\end{equation*}
Let $\operatorname{Irr}(G)$ be the set of all irreducible characters of $G$.  For each $\chi\in \operatorname{Irr}(G)$,
the symmetry classes of tensors associated with $G$ and $\chi$ is the image
of the symmetry operator
\begin{equation*}
\label{e11}
T(G,\chi)= \frac{\chi(e)}{|G|}\sum_{\sigma \in G}\chi(\sigma)P_\sigma,
\end{equation*}
and it is denoted by $V_{\chi}(G)$. We say that the tensor $T(G,\chi)(v_1\otimes\dots\otimes v_m )$ is a decomposable symmetrized tensor, and we denote it by $v_1\ast\dots\ast v_m$. The dimension of  $V_{\chi}(G)$ is given by
\begin{equation*}
\label{important1}
\dim (V_{\chi}(G))= \frac{\chi(e)}{|G|}\sum_{\sigma \in G}\chi(\sigma)n^{c(\sigma)},
\end{equation*}
where $c(\sigma)$ is the number of cycles, including cycles of length one, in the disjoint cycle factorization of $\sigma$; see  \cite {Me}.

The inner product on $V$ induces an $G$-invariant inner product on $V^{\otimes m}$ and hence on $V_\chi (G)$ which satisfies
$$
\langle v_1\ast\dots\ast v_m,u_1\ast\dots\ast u_m  \rangle= \frac{\chi(e)}{|G|}\sum_{\sigma \in G}\chi(\sigma)\prod_{i=1}^{m}\langle v_{i},u_{\sigma(i)}\rangle =\frac{\chi(e)}{|G|}d^G_\chi(A),
$$
where $A_{ij}=\langle v_i,u_j\rangle$ and $d^G_\chi$ is the generalized matrix function defined by $d^G_\chi(A):=\sum_{\sigma \in G}\chi(\sigma)\prod_{i=1}^{m}A_{i\sigma(i)}$.  If $v_i$'s and $u_j$'s are orthonormal vectors, then the formula becomes
\begin{equation}\label{sepprop}
 \langle e^{*}_\alpha,e^{*}_\beta \rangle =\left\{
                                             \begin{array}{ll}
                                               0 & \hbox{ if $\alpha \nsim \beta$,} \\
                                               \frac{\chi(1)}{|G|}\sum_{\sigma \in G_\beta}\chi(\sigma h^{-1}) & \hbox{ if $\alpha=\beta h$.}
                                             \end{array}
                                           \right.
\end{equation}
In particular, for $\sigma_1, \sigma_2 \in G $ and 	$\alpha \in \Gamma_{m,n}$, we have
\begin{equation}
\label{important2}
\langle e^{*}_{\alpha},e^{*}_{\alpha\sigma}\rangle = \frac{\chi(e)}{|G|}\sum_{h \in  G_\alpha }\chi(\sigma h).
\end{equation}
Since $\sigma e^*_\alpha = e^*_{\alpha\sigma^{-1}}$ for each $\sigma$ \cite[Lemma 1.3]{HK} and the induced inner product is $G$ invariant,
\begin{equation}\label{innnerinv1}
    \langle e^{*}_{\alpha\sigma_1},e^{*}_{\alpha\sigma_2}\rangle=\langle e^{*}_{\alpha},e^{*}_{\alpha\sigma_2\sigma_1^{-1}}\rangle,
\end{equation}
for each $\sigma_1,\sigma_2$ in $G$ and $\alpha\in \Gamma_{m,n}$.  Moreover, by (\ref{important2}),  $$ \parallel  e^*_\alpha\parallel^2 =\frac{\chi(e)}{|G|}  \sum_{h \in G_\alpha}\chi(h). $$
By setting
$$
\overline{\Delta} = \{\alpha \in \Delta \,|\, \sum_{h \in G_\alpha}\chi(h) \neq 0   \},
$$
we have  $e^{*}_{\alpha}\neq 0$ if and only if $\alpha \in \overline{\Delta}$.\\

For $\alpha \in \overline{\Delta} $, $ V^{*}_\alpha: =\langle e^{*}_{\alpha\sigma}\,|\, \sigma \in G  \rangle $ is called the orbital subspace of $ V_\chi(G).$ By (\ref{sepprop}), it follows that
$$
V_\chi(G)= \bigoplus_{\alpha \in \overline{\Delta}} V^{*}_\alpha
$$
is an orthogonal direct sum. In \cite{F2}, it is proved that
\begin{equation}
\label{important3}
\dim (V^{*}_\alpha) ~=~ \frac{\chi(e)}{|G_\alpha|}\sum_{h \in G_\alpha}\chi(h).
\end{equation}
Thus, if $\chi$ is linear, then $|\chi(h)|\leq 1$, for each $h\in G_\alpha$ and thus $\dim V^{*}_\alpha \leq \chi(e)=1$ and in this case, the set
$
\{e^{*}_{\alpha} \,|\, \alpha \in \overline{\Delta}   \}
$
is an orthogonal basis of $V_\chi(G)$. An orthogonal basis which consists of the decomposable symmetrized tensors $e^{*}_{\alpha}$ is called an \emph{orthogonal $\ast$-basis} or \emph{o*-basis} for short. If $\chi$ is not linear, it is possible that $V_\chi(G)$  has no o*-basis.   More details and motivation for the study of symmetry classes of tensors can be found in  \cite{Me}.

To prove the main result, we also need some facts about vanishing sum of roots of unity.  For a given natural number $m=p_1^{a_1}p_2^{a_2}\cdots p_t^{a_t}$, suppose that $\epsilon_1,\epsilon_2,...,\epsilon_k \in \mathbb{C}$ are $m$th roots of unity.  The main result of T.Y. Lam and K. H. Leung in \cite{Lam and Leung}, asserts that
\begin{equation}\label{Lamandleung}
   \hbox{\emph{ if  $k\notin \mathbb{N}_0\langle\{p_1,\dots,p_t \}\rangle$, then $\sum_{i=1}^k\epsilon_i\neq 0$},}
\end{equation}
where $ \mathbb{N}_0\langle\{p_1,\dots,p_t \}\rangle=\{ k_1p_1+\cdots+k_tp_t\,|\,  k_i\in\mathbb{N}_0 , \hbox{ for each $1\leq i \leq t$}\}$ and $\mathbb{N}_0$ is the set of nonnegative integers.

\section{Semi-direct product of finite abelian groups and Characters}
Let $A$, $H$  be finite groups with identity $e_A$, $e_H$, respectively.  Let $\phi: H\longrightarrow \operatorname{Aut}(A)$ be a group homomorphism and we denote $\phi(h)$ by $\phi_h$, for each $h\in H$.  Let $G=A\rtimes _\phi H$.  As a set, $G=\{ (a,h)\, |\, a\in A, h\in H\}$.  Operation on the group $G$ is given by $$(a_1,h_1) (a_2,h_2)=(a_1\phi_{h_1}(a_2), h_1h_2),$$ for each $(a_1,h_1)$ and $(a_2,h_2)$ in $G$.  The identity of the group $G$ is $(e_A,e_H)$ and the inverse of $(a,h)\in G$ is $(\phi_{h^{-1}}(a^{-1}), h^{-1})$.
Note that $\{(a,e_H)\, |\, a\in A  \}\cong A$ is a normal subgroup of $G$ and $\{(e_A,h)\, |\, h\in H  \} \cong H$ is a subgroup of $G$.  It is also well known that $G$ is a semi-direct product of $A$ by $H$ if and only if the short exact sequence $ \{e\}\longrightarrow A\longrightarrow G \longrightarrow H\longrightarrow \{e\} $ splits.  Moreover, by Schur-Zassenhaus theorem, if $G$ is a finite group and $N$ is a normal subgroup whose order is coprime to the order of the quotient $G/N$, then the short exact sequence $$ \{e\}\longrightarrow N\longrightarrow G \longrightarrow G/N\longrightarrow \{e\} $$ splits, i.e., $G$ is a semi-direct product of $N$ by $G/N$.\\

Let $\Omega$ be an $H$-set (with the action $\phi$) and let $K$ be the direct product $K=\prod_{\omega \in \Omega} A_\omega$ of copies of $A_\omega=A$ by the set $\Omega$.  The elements of K can be seen as sequences $(a_\omega)$ of elements of A indexed by $\Omega$ with componentwise multiplication. Then the action $\phi$ of $H$ on $\Omega$ extends in a natural way to an action $\phi_\Omega$ of $H$ on the group $K$ by $\phi_\Omega(h) (a_\omega) := (a_{h^{-1}\omega})$.   Then the wreath product of $A$ by $H$ indexed by $\Omega$ is defined by $K\rtimes _{\phi_\Omega} H$ and it is denoted by $$A \wr_\Omega H.$$  By the universal embedding theorem (known as the Krasner-Kaloujnine embedding theorem), if $G$ is an extension of $A$ by $H$, then there exists a subgroup of the wreath product $A \wr_\Omega H$ (for some $\Omega$) which is isomorphic to $G$.\\

 Now, suppose that $A$ is an abelian groups.  Then,  the set of all irreducible characters (all are linear) $\operatorname{Irr}(A):=A^\vee=\operatorname{Hom} (A,\mathbb{C})$ is isomorphic to $A$.  Thus, the action of $H$ on $A$ induces the action on $A^\vee$.  Precisely, for each $x\in A^\vee$ and $h\in H$, $h\cdot x:=\phi_h(x)=x\circ \phi_h$.  Let $[x]$ be the orbital and $H_x$ be the stabilizer of $x$ under this action.    Let $U$ be an irreducible representation of $H_x$.  We define a representation $V_{([x], U)}$ of $G=A\rtimes _\phi H$ as follows.  As a representation of $H$, we set
 $$ V_{([x],U)}:=\operatorname{Ind}_{H_x}^H (U)=\{ f:H \longrightarrow U \,|\, f(\sigma h)=\sigma f(h), \sigma \in H_x, h\in H\}.$$
 Next, we introduce an additional action of $A$ on $V_{([x],U)}$ by $(af)(h):=(x\circ \phi_h)(a) f(h)$.  The combination of these two actions becomes an action of $G=A\rtimes _\phi H$.  \\

  It can be shown that (see Theorem 4.75, \cite{POSTADE}, for example) this representation is independent from the choice of $x\in[x]$.  Also, if $\{ [x_1],\ldots, [x_k]\}$ is the set of all disjoint orbits for the action on $H$ on $A^\vee$, then  $$\{  V_{([x_i],U)} \,|\, U\in \operatorname{Irr}(H_{x_i}) ,i=1,2,...,k\}, $$ forms a complete set of irreducible representations of $G=A\rtimes_\phi H$.  The character $\chi_V$ of $V:=V_{([x],U)}$, is given by the Mackey-type formula: for each $a\in A$ and $g\in H$,
 \begin{equation}\label{Mackey formula}
    \chi_V(a,g)=\frac{1}{|H_x|} \sum_{h\in H : hgh^{-1}\in H_x}x(\phi_h(a))\chi_U(hgh^{-1}).
  \end{equation}
In particular, if both $A$ and $H$ are abelian, then formula (\ref{Mackey formula}) reduces to
\begin{equation}\label{Mackey formula2}
    \chi_V(a,g)=\left\{
                  \begin{array}{ll}
                     \frac{\chi_U(g)}{|H_x|} \sum_{h\in H}x(\phi_h(a)) & \hbox{if $g\in H_x$,} \\
                    0 & \hbox{if $g\notin H_x$.}
                  \end{array}
                \right.
  \end{equation}
\section{Results}
Let $G=A\rtimes _\phi H$, where both  $A,H$ are finite abelian groups, and  $\chi$ be an irreducible character of $G$.  Let $V$ be an inner product space with $\dim V=n>1$ and $V_\chi(G)$ be a symmetry class of tensors in the $m$-folds tensor space $V^{\otimes m}$. In the the following results, we denote by $\operatorname{Prime}(|A|)$ the set of all prime factors of $|A|$.
\begin{thm} \label{mainthm} If there is $\alpha \in \Gamma_{m,n}$ such that $G_\alpha=\{e\}$ and $|H|\notin \mathbb{N}_0\langle \operatorname{Prime}(|A|)\rangle$, then $V_\chi(G)$ admits o*-basis if and only if $\chi$ is linear.
\end{thm}
\begin{proof} We write $\chi$ for $\chi_{V_{([x], U)}}$, where $U\in \operatorname{Irr}(H_x)$.  If $\chi$ is linear then, by discussion in Section 2,  $V_\chi(G)$ admits o*-basis.  On the other hand, assume that there is an o*-basis for $V_\chi(G)$.  Then, $V^*_\alpha(G)$ has an o*-basis; say, $B=\{ e^*_{\alpha(a_1,h_1)},\dots,e^*_{\alpha(a_{s_\alpha},h_{s_\alpha})}\}$ for some $(a_i,h_i)\in G$, where ${s_\alpha}=\dim V^*_\alpha(G)$.

  Since $H$ is abelian, $U$ is linear.  Then, by (\ref{Mackey formula2}), $G_\alpha=\{e:=(e_A,e_H)\}$, and by Freese's theorem, we calculate that
  \begin{equation}\label{salpha}
s_\alpha=\frac{\chi(e_A,e_H)}{|G_\alpha|}\sum_{(a,h)\in G_\alpha}\chi(a,h)=\left( \frac{\chi_U(e_H)}{|H_x|} \sum_{h\in H}x(\phi_h(e_A)) \right)^2=(\frac{|H|}{|H_x|})^2.
  \end{equation}
  Moreover, in (\ref{Mackey formula2}), $\sum_{h\in H}x(\phi_h(a))$ is an $|H|$-terms sum of $|A|$th-roots of unity.  Since $|H|\notin \mathbb{N}_0\langle \operatorname{Prime}(|A|)\rangle$, by (\ref{Lamandleung}) and (\ref{Mackey formula2}),  the sum can not be zero.  We now conclude that, for each $(a,h)\in G$,
\begin{equation}\label{cricha1}
\chi(a,h)= 0 \hbox{ if and only if } h\notin H_x.
\end{equation}

We claim that $h_iH_x\cap h_j H_x=\emptyset$ if $i\neq j$.   Suppose $h_i=\sigma h_j$, for some $\sigma\in H_x$.   Since $e^*_{\alpha(a_i,h_i)}$ and $e^*_{\alpha(a_j,h_j)}$ are orthogonal and the inner product is $G$ invariant,
 \begin{eqnarray*}
  0&=& \left<e^*_{\alpha(a_i,h_i)}, e^*_{\alpha(a_j,h_j)}\right>  \\
  &=& \left<e^*_{\alpha}, e^*_{\alpha(a_j,h_j)(a_i,\sigma h_j)^{-1}}\right> \\
  &=& \left<e^*_{\alpha}, e^*_{\alpha(a_j,h_j)(\phi_{(\sigma h_j)^{-1}}(a_i^{-1}),(\sigma h_j)^{-1})}\right> \\
  &=& \left<e^*_{\alpha}, e^*_{\alpha(a_j\phi_{\sigma^{-1}}(a_i^{-1}),\sigma^{-1})}\right>\\
  &=&  \frac{\chi(e)}{|G|}\chi(a_j\phi_{\sigma^{-1}}(a_i^{-1}),\sigma^{-1})\neq 0,
 \end{eqnarray*}
 by (\ref{cricha1}), which is a contradiction and then the claim becomes true.  Therefore, $$|H|\geq s_\alpha |H_x|.$$
 By (\ref{salpha}), we conclude that $|H_x|=|H|$.  This implies that $$\deg(\chi)=\chi(e_A,e_H)=|H|/|H_x|=1$$ and hence $\chi$ is linear.
\end{proof}
Let $\Omega$ be a finite $H$-set.  Then $\operatorname{Prime}(|A|^{|\Omega|})=\operatorname{Prime}(|A|)$.  Also, if $A$ is abelian, then so is $K=\prod_{\omega \in \Omega} A_\omega$, where $A_\omega=A$, for each $\omega \in \Omega$.  Hence, the following result is an immediate fact from Theorem \ref{mainthm}.
\begin{cor}
Let $G=A \wr_\Omega H$, where $A, H$ are finite abelian groups.  If there is $\alpha \in \Gamma_{m,n}$ such that $G_\alpha=\{e\}$ and $|H|\notin \mathbb{N}_0\langle \operatorname{Prime}(|A|)\rangle$, then $V_\chi(G)$ admits o*-basis if and only if $\chi$ is linear.
\end{cor}
Condition for having an o*-basis of symmetry classes of tensors associated to the following type of groups can be obtained from the above theorem as well.
\begin{cor}
Let $V$ be an inner product space with $\dim V=n>1$.  The symmetry classe of tensors $V_\chi(G)$, associated to the groups $G$ below, admits o*-basis if and only if $\chi$ is linear.
\begin{enumerate}
  \item Dihedral group $D_{2s}$, where $s$ is odd.
  \item Non-abelian groups of order $pq$.
  \item $Z$-group $G$ that its order has two prime factors and $G_\alpha=\{e\}$ for some $\alpha\in \Gamma_{m,n}$.
\end{enumerate}
\end{cor}
\begin{proof}
(1)  The dihedral group $D_{2s}$ is a semi-direct product  of $C_s=\langle a\rangle$ by $C_2=\langle b\rangle$, i.e., $D_{2s}=C_s\rtimes_\phi C_2$, where $\phi:C_2\longrightarrow \operatorname{Aut}(C_s)$ is defined by $\phi_b(a)=a^{-1}$.  Note that $|C_2|=2\notin \mathbb{N}_0\langle\operatorname{Prime}(|C_s|)\rangle$ (because $s$ is odd), and $G_\alpha=\{e\}$, where $\alpha=(1,2,1,...1)\in \Gamma_{s,n}$ (cf. \cite{HT}).  By Theorem \ref{mainthm}, the conclusion is immediate.  \\

(2) Each non-abelian group $G$ of order $pq$, where $q$  is prime and $p$ divides $q-1$, is a semi-direct product of a cyclic group $C_q=\langle a\rangle$ by a cyclic group $C_p=\langle b\rangle$, i.e., $G=C_q\rtimes _\phi C_p$ such that $\phi:C_q \longrightarrow \operatorname{Aut}(C_q)$ is given by $\phi_b(a)=a^r$, where $r$ is a primitive root of the congruence $z^p\equiv 1 (\operatorname{mod} q)$, \cite{SKBer}.  It is clear that $|C_p|=p\notin \mathbb{N}_0\langle\operatorname{Prime}(|C_q|)\rangle$ and $G_\alpha=\{e\}$, where $\alpha=(1,2,1,\dots,1))\in \Gamma_{q,n}$, \cite{kijtipq}.   Now, the result is immediate by Theorem \ref{mainthm}. \\

(3)  Note that $Z$-group $G$ is an extension group of a split short exact sequence of a cyclic group $C_s$ by a cyclic group $C_t$, whose orders are relatively primes, \cite{Wong}.  So, it is a semi-direct product $G=C_s\rtimes_\phi C_t$, for some $\phi\in \operatorname{Hom}(C_t, \operatorname{Aut}(C_s))$.  By the assumptions, all requirements of the Theorem \ref{mainthm} are satisfied and thus the result follows.

\end{proof}
In the following results, let $G$ be any finite group.  Let $\chi$ be an irreducible character of $G$, $V_\chi(G)\leq V^{\otimes m}$, and denote $Z_\chi:=\{g\in G\, |\, \chi(g)=0\}$. Using the same idea as in the proof of Theorem \ref{mainthm}, we have:
\begin{thm}  Suppose that there is $\alpha \in \Gamma_{m,n}$ such that $G_\alpha=\{e\}$. If there exists a subgroup $H$ of $G$ such that $H\subseteq G\setminus Z_\chi$ and $[G:H]<\chi(e)^2$, then $V_\chi(G)$ does not admit an o*-basis.
\end{thm}
\begin{proof} Since $G_\alpha=\{e\}$, $s_\alpha=\dim (V^*_\alpha(G))=\chi(e)^2$.  Suppose that $B=\{e^*_{\alpha g_1},\dots,e^*_{\alpha g_{s_\alpha}} \}$ is an o*-basis of $V^*_\alpha(G)$. We claim that $Hg_i\cap Hg_j=\emptyset$ if $i\neq j$.  Suppose for a contradiction that $h_1g_i=h_2g_j$, for some $1\leq i\neq j\leq s_\alpha$ and some $h_1, h_2\in H$.   Then, $g_jg_i^{-1}=h_2^{-1}h_1$ and thus
$$
  0 = \left<e^*_{\alpha g_i},e^*_{\alpha g_j} \right> = \left<e^*_{\alpha },e^*_{\alpha g_jg_i^{-1}} \right> =  \left<e^*_{\alpha },e^*_{\alpha h_2^{-1}h_1} \right>=\frac{\chi(e)}{|G|}\chi(h_2^{-1}h_1) \neq0,
$$
by (\ref{important2}), (\ref{innnerinv1}) and $h_2^{-1}h_1 \in H\subseteq G\setminus Z_\chi$.   Hence, $|G|\geq s_\alpha |H|$, which completes the proof.
\end{proof}
\section*{Acknowledgements}

Most parts of this paper were done while the first author was the visiting assistant professor in the Department of Mathematics and Statistics, Auburn University, AL, USA.  He would like to express his thanks to Prof. Tin-Yau Tam for the hospitality and his suggestion on the earlier draft.

\bigskip

\address \textbf{Kijti Rodtes} \\

{ Department of Mathematics, Faculty of Science, \\ Naresuan University, and Research Center for Academic Excellent in Mathematics \\ Phitsanulok 65000, Thailand}\\
\email{kijtir@nu.ac.th, \quad \quad and}\\

{ Department of Mathematics and Statistics, \\Auburn University, Alabama 36849, USA}\\
\email{kzr0033@auburn.edu}\\

\address \textbf{Kunlathida Chimla} \\

{ Department of Mathematics, Faculty of Science,\\ Naresuan University, Phitsanulok 65000, Thailand}\\
\email{kunlathida\_ nu@hotmail.com }\\


\begin{thebibliography}{20}
\bibitem{SKBer}
S. K. Berberian, \emph{Non-abelian groups of order pq}, Amer. Math. Monthly 60 (1953), pp. 37-40.
\bibitem{BPR}
C. Bessenrodt, M. R. Pournaki, and A. Reifegerste, \emph{A note on the orthogonal basis of a certain full symmetry class of tensors}, Linear Algebra Appl. 370 (2003), pp. 369--374.
\bibitem{Chi}
C. K. Li and A. Zaharia, \emph{Induced operators on symmetry classes of tensors}, Trans. Amer. Math. Soc. 354 (2002), pp. 807-836.
\bibitem{DP}
M. R. Darafsheh and  M. R. Pournaki, \emph{On the orthogonal basis of the symmetry classes of tensors associated with the dicyclic group}, Linear and Multilinear Algebra 47 (2000), pp. 137--149.
\bibitem{POSTADE}
P. Etingof, O. Golberg, S. Hensel. T. Liu, A. Schwendner, D. Vaintrob and E. Yudovina,  \emph{Introduction to representation theory} (MIT Open Courseware, 2011), http://math.mit.edu/~etingof/replect.pdf
\bibitem{F2}
R. Freese, \emph{Inequalities for generalized matrix functions based on arbitrary characters}, Linear Algebra Appl. 7 (1973), pp. 337--345.
\bibitem{H}
R. R. Holmes, \emph{Orthogonal bases of symmetrized tensor spaces}, Linear and Multilinear Algebra 39 (1995), pp. 241--243.
\bibitem{Holmm}
R. R. Holmes, \emph{Orthogonality of cosets relative to irreducible character of finite groups}, Linear and Multilinear Algebra 52, (2004), pp. 133-143.
\bibitem{HT}
R. R. Holmes and T. Y. Tam, \emph{Symmetry classes of tensors associated with certain groups}, Linear and Multilinear Algebra 32 (1992), pp. 21--31.
\bibitem{HK}
R. R. Holmes and A. Kodithuwakku, \emph{Orthogonal bases of Brauer symmetry classes of tensors for the dihedral group}, Linear and Multilinear Algebra  61 (2013), pp. 1136--1147.
\bibitem{HORO}
M. Hormozi and K. Rodtes, \emph{ Symmetry classes of tensors associated with the semi-dihedral groups $SD_{8n}$},
Colloquium Mathematicum 131 (2013), 59--67.
\bibitem{Lam and Leung}
T. Y. Lam and K. H. Leung, \emph{On vanishing sums of roots of unity}, J. Algebra 224 (2000), pp. 91-109.
\bibitem{Me}
R. Merris,\emph{ Multilinear Algebra}, Gordan and Breach Science Publishers, Amsterdam, 1997.
\bibitem{kijtipq}
K. Rodtes, \emph{Symmetry classes of tensors associated to nonabelian groups of order pq}, Bull. Aust. Math. Soc. 94 (2016),  pp. 36-42.
\bibitem{SAJ}
M. A. Shahabi, K. Azizi and M. H. Jafari, \emph{On the orthogonal basis of symmetry classes of tensors}, J. Algebra 237 (2001), pp. 637--646.
\bibitem{shahryari}
M. Shahryari, \emph{On the orthogonal bases of symmetry classes}, J. Algebra 220 (1999), pp. 327-332.
\bibitem{Gong}
B. Y. Wang and M. P. Gong, \emph{A higher symmetry class of tensors with an orthogonal basis of
       decomposable symmetrized tensors}, Linear and Multilinear Algebra 30 (1991), pp. 61-64.
\bibitem{Wong}
W. J. Wong, \emph{On finite groups with semi-dihedral Sylow 2 subgroups}, J. Algebra 4 (1966), pp. 52-63.
\end{thebibliography}
\end{document}